\documentclass[12pt]{amsart}
\usepackage{amssymb}
\usepackage{amsfonts}% lo he annadido
\usepackage{graphicx}
\usepackage{epsfig}
\usepackage{amssymb}
\usepackage{amsmath}
\usepackage{amsfonts}
\usepackage{amscd}
\usepackage{amsfonts}
\usepackage{enumerate}
\usepackage{amsmath, amsthm}
\usepackage{amscd}
\usepackage{amssymb}
\usepackage{psfrag}
\usepackage{pstricks}
\usepackage{latexsym,graphicx}
\usepackage{varioref}

\theoremstyle{plain}
\newtheorem{theorem}{Theorem}[section]
\newtheorem{proposition}[theorem]{Proposition}

\theoremstyle{definition}

\newcommand{\R}{\mathbb{R}}

\newcommand{\ra}{{\rightarrow}}

\let\cal\mathcal

\begin{document}
\title[Degeneration]
        { Degeneration of strictly convex real projective structures on surface }
        \author{Inkang Kim}
        \date{}
        \maketitle

\begin{abstract}
In this paper we study  the degeneration of  convex real
projective structures on  bordered surfaces.
\end{abstract}
 \footnotetext[3]{The author
gratefully acknowledges  the partial support
of grant  (NRF-2017R1A2A2A05001002).}

\section{Introduction}
Recently there has been intensive study of Anosov representations of Gromov hyperbolic groups into higher rank semisimple Lie groups extending the notion of convex cocompact representations in rank one semisimple Lie groups. On the other hand, for free groups, there is a notion of primitive stable representations proposed by Minsky \cite{Minsky}. This notion works well for representations with parabolic elements. It is known that the restriction of a primitive stable representation to a non-trivial free factor is Anosov \cite{KP}. Hence a primitive stable representation is a very good generalization of Anosov representations, yet not all representations with parabolics are primitive stable. For example, a finite volume strictly convex real projective surface with more than two cusps has a holonomy representation which is not primitive stable.
This can be easily seen  by choosing a free generating set containing one of the elements representing a hole. Yet a finite volume strictly convex real projective surface with one cusp has a primitive stable holonomy representation \cite{KK}.

Despite the intensive study of Anosov representations, its boundary of the space of Anosov representations is not well-understood. For example, for the quasifunchsian space, it is known that the cusps are dense on the boundary of quasifuchsian space. This is no longer true for Anosov representations.
Take the space $\cal A_2$ of Anosov representations from free group on 2 generators into $SL(3,\R)$.
Then $\cal A_2$ contains the space $\cal P$ of marked convex real projective structures on a pair of pants $P$ with hyperbolic geodesic boundaries.
It also contains the space $\cal T$ of marked  convex real projective structures on a puctured torus $T$ with a hyperbolic geodesic boundary. The real dimension of $\cal A_2$ is $-8\chi(T)=-8\chi(P)=8$. Since both $\cal T$ and $\cal P$ have real dimension 8, they are open subsets of $\cal A_2$.
The space $\cal T_c$ consisting of convex real projective structures  on $T$ whose boundary curve is parabolic has real dimension 6. By cutting  $T$ along one simple closed curve $C$, we can see
that any structure in $\cal T_c$ has two boundary parameters on $C$, two gluing parameters on $C$
and two internal parameters on $T\setminus C$. This shows that  $\cal T_c$ is included as a boundary of $\cal A_2$ of codimension 2. Similarly if we pinch one of the boundary of $P$ to a parabolic, then the real dimension drops by 2. This suggests that the cusps appear in the codimension 2 strata. Hence cusps are not dense in $\partial \cal A_2$.
%If one pinches $C$ and the boundary of $T$, and pinches all the boundary components of $P$, both %surfaces become thriced 
%punctured strictly convex real projective surface with 3 cusps. The dimension of this space $\cal S$ is 2. %Hence
%$\overline{\cal P}$ and $\overline{\cal T}$ share $\cal S$ on the boundary of $\cal A_2$.

Such a phenomenon seems generic in higher rank case and it manifests a sharp contrast to the rank one case.
In this paper, we are interested in such  degeneration phenomena in terms of Fock-Goncharov or Bonahon-Dreyer coordinates.

In this note, the following theorem is proved in Sections \ref{moduli} and \ref{degeneration}. An explicit conventions and notations can be found in Section \ref{degeneration}.
\begin{theorem} The space of Anosov representations $\cal A_2$ contains $\cal P$ and $\cal T$ as components, and the boundaries of $\cal P$ and $\cal T$ in the character variety $\chi_2$ can be described explicitly using Bonahon-Dreyer parameters when the holonomy of the boundary loops degenerate to parabolic or quasi-hyperbolic elements.\\
(i) For the boundary of $\cal P$
\begin{enumerate}
\item The space of convex projective structures on $P$ with one boundary, say $A_1$, being parabolic is parametrized by $$\sigma_1(B_1),\sigma_2(B_1),\sigma_1(B_2),\sigma_1(B_3),\tau_{111}(T_+),\tau_{111}(T_-).$$
\item The space of convex projective structures on $P$ with all boundaries being parabolic is parametrized by
$\sigma_1(B_1),\tau_{111}(T_+)$.
\item The space of convex projective structures on $P$ with a boundary $A_1$ being quasi-hyperbolic is parametrized by 
$$\sigma_1(B_1),\sigma_2(B_1),\sigma_1(B_2),\sigma_2(B_2),\sigma_1(B_3),\tau_{111}(T_+),\tau_{111}(T_-)   .$$
\end{enumerate}
(ii) For the boundary of $\cal T$
\begin{enumerate}
\item The space of convex projective structures on $T$ with boundary being parabolic is parametrized by
$$\sigma_1(B_1),\sigma_1(B_2),\sigma_1(B_3),\tau_{111}(T_+), \sigma_1(C),\sigma_2(C)$$ where $C$ is a meridan curve so that $T\setminus C$ is a pair of pants.
\item The space of convex projective structures on $T$ with boundary being quasi-hyperbolic is parametrized by
$$\sigma_1(B_1),\sigma_1(B_2),\sigma_1(B_3),\tau_{111}(T_+), \tau_{111}(T_-), \sigma_1(C),\sigma_2(C).$$
\end{enumerate}
\end{theorem}

We also address the relationship between area and divergence;
\begin{theorem} If a triangle invariant or a difference of two shear coordinates along some geodesic diverges, then the Hilbert area of the convex projective surface tends to infinity.
\end{theorem}
This theorem is proved in Section \ref{area}.
\vskip .1 in
{\bf Acknowledgements} The author thanks F. Bonahon for the conversations on projective structures and parametrizations during his visit to USC for several occasions. He also thanks S. Tillmann for his interests and
pointing out relevant references.
\section{Preliminaries}
\subsection{Isometries in Hilbert metric}
Suppose that $\Omega$ is a (not necessarily strictly) convex domain in
$\mathbb{RP}^2$. Choose an affine set $A$ containing $\Omega$ with a Euclidean norm
$| \cdot |$. For $x\neq y\in \Omega$, let $p,q$ be the
intersection points of the line $\overline{xy}$ with $\partial\Omega$ such that $p,x,y,q$ are in this order. The
{\bf Hilbert distance} is defined by
$$d_\Omega(x,y)=\frac{1}{2}\log \frac{|p-y||q-x|}{|p-x||q-y|}$$ where $| \cdot |$ is a Euclidean
norm. This metric coincides with the hyperbolic metric if
$\partial\Omega$ is a conic.
We introduce the notion of hyperbolic, quasi-hyperbolic and parabolic isometries as they will appear in the sequel.
An isometry is {\bf hyperbolic} if it can be represented by an element  with diagonal entries $\lambda_1>\lambda_2>\lambda_3>0$ in $SL(3,\R)$. It is {\bf quasi-hyperbolic} if 
it is conjugate to
$\begin{bmatrix}
 \mu & 1 & 0 \\
  0     & \mu & 0 \\
0 & 0 &  \nu\end{bmatrix}$ where $\mu>\nu>0, \mu^2\nu=1$. Finally  it is parabolic if it is conjugate to  $\begin{bmatrix} 
           1 & 1 & 0 \\
            0              & 1 & 1\\
            0           & 0 & 1\end{bmatrix}.$
\subsection{Bonahon-Dreyer parameters}
A flag in $\R^3$ is a family $F$ of subspaces $F^0\subset F^1 \subset F^2$ where $F^i$ has dimension $i$.
A pair of flags $(E,F)$ is generic if every $E^i$ is transverse to every $F^j$. Similarly a triple of flags $(E,F,G)$ is generic if each triple $E^i, F^j, G^k$ meets transversely.

 For such a generic triple, one defines a triangle invariant
$$T(E,F,G)=\frac{e^2\wedge f^1}{f^1\wedge g^2}\frac{ e^1\wedge g^2}{e^1\wedge f^2}\frac{ f^2\wedge g^1}{  e^2\wedge g^1},$$ where $0\neq e^i\in \wedge^i(E^i)$ and similar for $f^j, g^k$. Note that for $i+j=3$, $e^i \wedge f^j$ can be
identified with a real number by taking a determinant of corresponding $3\times 3$ matrix $(e^i, f^j)$.
Easy calculation shows that
$$ T(E,F,G)=T(F,G,E)=T(F,E,G)^{-1}.$$

Let $(E,F,G,L)$ be a generic flags in the sense that every quadruple $E^a,F^b, G^c, L^d$ meets transversely.
In this case we have two double ratios
$$ D_1(E,F,G,L)=-\frac{e^1\wedge f^1\wedge g^1}{e^1\wedge f^1 \wedge l^1}\frac{f^2\wedge l^1}{f^2\wedge g^1}$$
$$D_2(E,F,G,L)=-\frac{e^2\wedge g^1}{e^2\wedge l^1}\frac{e^1\wedge f^1\wedge l^1}{e^1\wedge f^1\wedge g^1}.$$

Let $\rho:\pi_1(\Sigma)\ra SL(3,\R)$ be a holonomy representation which gives rise to a  convex real projective structure on  $\Sigma$.  Give an ideal triangulation  to $\Sigma$ coming from the ideal triangulation of
each pair of pants in a fixed pants decomposition of $\Sigma$. Let $T_j$ be an ideal triangle and $\tilde T_j$ a lift
to $\tilde \Sigma$.  Let $\tilde v_j, \tilde v_j', \tilde v_j''\in\partial_\infty \tilde \Sigma$ be vertices of $\tilde T_j$ in clockwise
order.  Define the triangle invariant of $T_j$ to be
$$\tau_{111}^{\rho}(T_j, v_j)=\log T(\cal F_\rho(\tilde v_j), \cal F_\rho(\tilde v_j'), \cal F_\rho(\tilde v_j'')).$$
Here the flat $\cal F_\rho(\tilde v_j)$ is determined by the tangent line at $\tilde v_j$ to $\partial \Omega_\rho$ where $\Omega_\rho$ is a convex domain in $\mathbb{RP}^2$ determined by $\rho$.

Let $\tilde \lambda$ be an oriented leaf in $\tilde \Sigma$ such that two ideal triangles $\tilde T$ and $\tilde T'$ share
$\tilde \lambda$ such that $\tilde T$ lies on the left of $\tilde \lambda$. Let $x$ and $y$ be the positive and negative end points of $\tilde\lambda$ and $z,z'$ be third
vertices of $\tilde T$ and $\tilde T'$ respectively. Then $i$-th shear invariant along the oriented leaf $\lambda$ is
$$\sigma^\rho_i(\lambda)=\log D_i(\cal F_\rho x, \cal F_\rho y, \cal F_\rho z, \cal F_\rho z').$$

\subsection{Goldman parameters for a pair of pants}
Let $A,B,C$ denote the boundary of a pair of pants satisfying $CBA=I$. By abuse of notations, $A,B,C$ denote isometries corresponding to the curves, or repelling fixed points of isometries on the universal cover.  We adopt the notation that the eigenvalues
of the matrix satisfy that $\lambda_1>\lambda_2>\lambda_3$ whose product equal to 1 and
$$\ell_1=\log\lambda_1-\log\lambda_2,\ \ell_2=\log\lambda_2-\log\lambda_3.$$
Then the Hilbert length of a closed geodesic represented by $A$ is equal to $\ell_1(A)+\ell_2(A)$.
In Goldman's notation \cite{Goldman}, $\lambda(A)=\lambda_3(A)$ and $\tau(A)=\lambda_1(A)+\lambda_2(A)$.

Goldman showed that the convex structure on a pair of pants is determined by $(\lambda(A), \tau(A)), (\lambda(B),\tau(B)), (\lambda(C),\tau(C))$ plus two internal parameters $s$ and $t$.
For the precise definition of $s$ and $t$, it is advised to  consult \cite{BK} because some notations and conventions
between \cite{BK} and \cite{Goldman} are different.

\subsection{Bulging deformation}
Even if all the triangle invariants are bounded, if the difference of shear coordinates along some infinite edge of an ideal triangle  goes to infinity, then the Hilbert area tends to infinity as we will see in Proposition \ref{bulging}.
We recall a {\bf bulging deformation} along a geodesic. 

Let $\Omega$ be a convex domain and $l$ an oriented infinite geodesic in $\Omega$.  Draw two tangent lines to
$\Omega$ at $l(\pm\infty)$ and denote $l^\perp$ the intersection of the two tangent lines. Then the bulging deformation along
$l$ is stretching the domain toward $l^\perp$. See \cite{FK} for details. More precisely,  let $l(-\infty)$  be (1,0,0),  $l(\infty)$ be $(0,0,1)$ and  $l^\perp$ be $(0,1,0)$. Use the base $\{(1,0,0), (0,1,0), (0,0,1)\}$ to write a projective transformation in a matrix form. Then the bulging along $l$ corresponds to a matrix
$$\begin{pmatrix}
 e^{-v} & 0 & 0 \\
 0 & e^{2v} & 0 \\
 0 & 0 & e^{-v} \end{pmatrix}. $$   Let $(1,y, x)$ be the vertex of an ideal triangle whose one edge is $l$
and on the right side of $l$.
Under the bulging deformation along $l$, this vertex moves to $(1,e^{3v}y, x)$ and two shear parameters along $l$
under the bulging deformation changes:
$$\sigma_1 \ra \sigma_1 -3v, \sigma_2\ra \sigma_2 +3v.$$
Hence the difference $\sigma_2-\sigma_1$ measures the bulging parameter $6v$.

\section{The moduli space of Anosov representations and the space of convex projective structures}\label{moduli}
\subsection{Component of the moduli space of  Anosov representations}
As we  have already seen in the introduction, the space $\cal A_2$ of Anosov representations from the free group $F_2$ to $SL(3,\R)$ contains two disjoint spaces, $\cal P$ the space of convex real projective structures with geodesic boundary on the pair of pants $P$, and $\cal T$ the space of convex real projective structures with a geodesic boundary on the punctured torus $T$.  Here by Anosov we means  Borel-Anosov where Borel subgroup is the set of upper triangular matrices.
As we have seen in a previous section, $\cal P$ has a strata of boundary of codimension 2 by pinching each boundary to cusps.  On the other hands, $\cal T$ has a boundary of codimension 2, $\cal T_c$, where the  boundary of $T$  is a cusp.

It is obvious from the construction that $\overline{\cal P}$ and $\overline{ \cal T}$ are disjoint where the closures are taken in the space of Anosov representations. Since both
$\cal P$ and $\cal T$ are cells of dimension 8, they form two disjoint components of $\cal A_2$.
The openess of $\cal P$ and $\cal T$ follow from the explicit  coordinates.

In this section let $\Sigma$ be a bordered surface of genus $g$ with $n>0$-puncture.
Then the set $\cal C$ of marked convex projective structures on $\Sigma$ with hyperbolic geodesic boundaries form an open set in the moduli space of Anosov representations $\cal A$.
 To show the closedness of $\cal C$,
suppose $\rho_i:\pi_1(\Sigma)\ra SL(3,\R)$ corresponding to convex projective structures in $\cal C$, converges to $\rho$ in $\cal A$.

Let $X=SL(3,\R)/SO(3)$ and
fix a generating set $S$ of $\pi_1(\Sigma)$.  For a representation $\rho:\pi_1(\Sigma)\ra SL(3,\R)$, set $d_\rho(x)=\sup_{s\in S} d(x,\rho(s)x)$ and $\mu(\rho)=\inf_{x\in X}d_\rho(x)$.
Let $\text{Min}_\rho=\{x\in X|d_\rho(x)=\mu(\rho)\}$. This is a closed convex set of $X$.

It is not difficult to see that $\text{Min}_\rho$ is non-empty and bounded if and only if $\rho$ is not parabolic, i.e., $\rho(\pi_1(\Sigma))$ does not fix a point in $\partial X$. Indeed, if $\text{Min}_\rho$ is either empty or unbounded, one can find a sequence $x_i\ra \eta \in\partial X$ so that 
$$d(x_i, \rho(s)x_i)\leq \mu(\rho)+\epsilon$$ for all large $i$ and $s\in S$. This implies that $\eta$ is fixed by $\rho$ and hence $\rho$ is parabolic.

Let's fix a base point $x_0\in X$. By conjugating $\rho_i$ if necessary, we may assume that $x_0\in \text{Min}_{\rho_i}$. Note that $\text{Min}_{\rho_i}$ is not empty since $\rho_i$ is not parabolic. Since $\lim\sup\{\mu(\rho_i)=d_{\rho_i}(x_0)\}\leq\mu(\rho)$, one
can extract a subsequence converging to $\tau$. Then $\tau$ is again discrete and faithful. See \cite{GM} for example. 

 Suppose $\tau(\pi_1(\Sigma))$ is contained in a parabolic subgroup $G_\eta$ of $SL(3,\R)$. Here $G_\eta$ means a parabolic subgroup stabilizing $\eta$ in the visual boundary of $X$. In this case, $\text{Min}_\tau$ is either $\emptyset$ or unbounded.
Since $d_{\rho_i}$ converges uniformly to $d_\tau$, $x_0\in \text{Min}_\tau$, hence
$\text{Min}_\tau$ is unbounded. In this case one can show that  $\tau$ fixes two end points of some geodesic $l$ as follows. Let  $l$ be a geodesic emanating from $\eta$
so that $\text{Min}_\tau\cap l=L$ is noncompact. This is possible since $\text{Min}_\tau$ accumulates to $\eta$. We want to show that $L=l$. Choose $y\in L$ and take a generalized Iwasawa decomposition  of the parabolic subgroup $G_\eta$ as 
$G_\eta=N_\eta A_\eta K_\eta$ where $K_\eta$ is an isotropy subgroup of $K$, the maximal compact subgroup stabilizing $y$ ( indeed $K_\eta$ fixes $l$ pointwise), $A_\eta y$ is the union of parallels to $l$, and $N_\eta$ is the horospherical subgroup which is determined uniquely by $\eta$. See Proposition 2.17.5 (4) in \cite{Eberlein}.
Note for any $id\neq n\in N_\eta$, $nl$ and $l$ are not parallel, but they are asymptotic at $\eta$. Indeed
$d(l(t), nl(t))\ra 0$ as $\lim_{t\ra -\infty} l(t)=\eta$. This fact implies the following. For any $g\in \tau(S)$, where $S$ is the fixed generating set of $\pi_1(\Sigma)$, if $g=nak$, we claim that $n=id$. If not,
$d(nak(l(t)), l(t))=d(na(l(t)), l(t))$ is strictly decreasing as $t\ra -\infty$ since $al(t)$ is a geodesic emanating from $\eta$. Since $l(t)\in L$ for large negative $t$, this contradicts the definition of $\text{Min}_\tau$. Hence
any element in $\tau(S)$ sends $l$ to a parallel geodesic. Consequently, $\tau(\pi_1(\Sigma))$ fixes two end points of $l$ and $\tau(\pi_1(\Sigma))\subset A_\eta K_\eta$.

Let $W$ be
the union of parallels to $l$, which is isometric to $l\times Y$ where $Y$ is a closed convex complete
subset of X. See Lemma 2.4 of \cite{BGS}.

If $l$ is non-singular, $W$ is a maximal flat containing $l$. If we take a Iwasawa decomposition
$SL(3,\R)=KAN$ where $K$ is the stabilizer of $x_0$, $Ax_0$ is the maximal flat containing $l$, $N$ is the
Nilpotent subgroup fixing $l(-\infty)$, then $\tau(\pi_1(\Sigma))\subset MA$ since $\tau(\pi_1(\Sigma))$ stabiizes $W$, which is abelian. Since
$\pi_1(\Sigma)$ is free and $\tau$ is faithful, this is impossible.

If $l$ is singular, $W=l\times H^2$.  Since $\tau$ preseves this splitting, it can be conjugated so that
its image is contained in 
$$
\begin{pmatrix}
 M_{2\times 2} & 0 \\
      0                 & \lambda \end{pmatrix}.$$
Then the projection $$\pi:\tau(\pi_1(\Sigma))\ra Iso(H^2)=SL(2,\R), \begin{pmatrix} M & 0 \\
                                                            0 & \lambda\end{pmatrix}\ra \frac{1}{(det M)^{1/2}} M$$ is either discrete or solvable
by Proposition 7.2.2 of \cite{Eberlein}. The kernel of $\pi$ will be central in $\tau(\pi_1(\Sigma))$ since it is of the form $\begin{pmatrix}
 I  & 0 \\
 0 & \lambda \end{pmatrix}$, and since $\tau$ is faithful and $\pi_1(\Sigma)$ has no center, $\pi$ is injective.
Hence $\pi\tau(\pi_1(\Sigma))$ is a discrete and faithul representation into $SL(2,\R)$.
Note that $\tau([\pi_1(\Sigma),\pi_1(\Sigma)])$ has image in $\begin{pmatrix}  SL(2,\R) & 0 \\
                   0 & 1\end{pmatrix}$. Hence $\tau|_{[\pi_1(\Sigma),\pi_1(\Sigma)]}$ is a discrete and faithful representation into $SL(2,\R)$.  Then by Lemma 2 of \cite{CG}, there exists an element $\gamma\in [\pi_1(\Sigma),\pi_1(\Sigma)]$ such that $tr(\tau(\gamma))<0$, i.e., its eigenvalues are
$(-\lambda_1,-\lambda_2,1)$ with $\lambda_i>0$.  Since $\rho_i\ra\tau$,  $\rho_i(\gamma)$ and $\tau(\gamma)$ will have the same sign for their eigenvalues for large $i$. But since $\rho_i$ is a positive representation in the sense of Fock-Goncharov \cite{FG}, the eigenvalues of $\rho_i(\gamma)$ must be all positive, a contradiction. This shows that $\tau(\pi_1(\Sigma))$ cannot be contained in a parabolic subgroup.

Suppose $\tau(\pi_1(\Sigma))$ is not contained in a parabolic subgroup. Then it is an irreducible representation in $\R^3$. Let $\Omega_i$ be a convex set in $\mathbb{RP}^2$ whose quotient by $\rho_i(\pi_1(\Sigma))$ is the corresponding real projective surface. Take a Hausdorff limit $\Omega$ of $\Omega_i$, which is invariant under $\tau$. Then $\Omega$ should be properly convex, otherwise $\tau$ will be reducible.
Hence $\Omega/\tau(\pi_1(\Sigma))$ is a properly convex projective surface.

Since $\rho_i\ra\rho$ before conjugating $\rho_i$, $[\rho_i]\ra [\rho]=[\tau]$.
If $\tau$ belongs to  $\partial \cal C$, some boundary of $\Sigma$ must be either  parabolic or quasi-hyperbolic. For parabolic case,  the orbit in $X$ under such a parabolic element is not quasi-geodesic, which is not possible for Anosov representation. For quasi-hyperbolic case,  if $L$ is an axis of the quasi-hyperbolic isometry on the boundary of $\Omega$, which connects the repelling point $p_-$ and the attracting point $p_+$, then the tangent line at $p_-$ contains the axis, while the tangent line at $p_+$ is different from the axis (see \cite{Marquis}). Hence the limit curve determined by $\Omega$ does not satisfy the antipodality that the Anosov representation should have. 
 In either case, $\rho$ is not Anosov. Since $[\rho]$ belongs to $\cal A$ by assumption, $[\tau]\in\cal C$.

This argument shows that:
\begin{theorem} Let $\Sigma$ be  a punctured surface and $\cal A$  a moduli space of Anosov representations from $\pi_1(\Sigma)$ to $SL(3,\R)$. Then $\cal C$, the set of marked convex projective structures on $\Sigma$ with hyperbolic geodesic boundaries, is closed in $\cal A$. Hence $\cal C$ is a  component of $\cal A$.
\end{theorem}

\subsection{Boundary of $\cal C$ in character variety}
We saw in the previous section that $\cal A$ contains a component $\cal C$ consisting of convex real projective structures on $\Sigma$ with hyperbolic boundaries. But $\cal A$ is just a subset of the character variety $\chi_n$ of free group with $n$ generators in $SL(3,\R)$.
In this section, we want to study $\partial \cal A$ in $\chi_n$. By a theorem of Bonahon-Dreyer \cite{BD}, 
$\cal C$ is a convex polytope in Euclidean space.

\begin{theorem}The  boundary of $\cal C$ consists of holonomies of real convex 
structures on $\Sigma$ with either quasi-hyperbolic boundary or parabolic boundary. Each quasi-geodesic boundary determines a codimension 1 boundary in the polytope, and each parabolic boundary determines
a codimension 2 boundary in the polytope.
\end{theorem}
\begin{proof}Suppose $[\rho_i] \in \cal C$ converge to a point $x$ in the finite boundary of the polytope. Then corresponding Fock-Goncharov coordinates are all finite for $x$. Then it is easy to see that all the traces
of boundary curves of  pairs of pants in pants decomposition of $\Sigma$ are bounded. Since the trace of any other curve
can be generated by the traces of these pants curve, $x$ should correspond to a finite point in the character variety, i.e.
$x$ is represented by a representation $[\rho]$.

The proof goes in  the exactly same way as in the previous section. One can show that, after conjugating $\rho_i$, $\rho_i\ra\tau$ and $\tau(\pi_1(\Sigma))$  is not contained in a parabolic subgroup. Furthermore $\Omega_i\ra\Omega$ in Hausdorff topology and $\Omega/\tau(\pi_1(\Sigma))$ is properly convex.
This shows that $[\rho]=[\tau]$ is represented by a holonomy of a real convex structure. Hence the boundary curve should be either quasi-hyperbolic or parabolic.
\end{proof}
As will be described in Theorem \ref{allparabolic},  the space of  convex real projective structures with all three boundary components of $P$ being parabolic is parametrized by $(\sigma_1(B_1), \tau_{111}(T_+))$. Hence this codimension 6 space is homeomorphic to $\R^2$.

The space of convex real projective structures on $T$ with the boundary curve being parabolic is
parametrized by six parameters $\sigma_1(B_i), i=1,2,3, \tau_{111}(T_+), \sigma_1(C),\sigma_2(C)$ where $C$ is a meridian curve so that $T\setminus C$ is a pair of pants. This space forms a codimension 2
boundary of $\cal T$ in $\chi_2$. Hence this codimension 2 boundary set is homeomorphic to $\R^6$.
\section{Area and degeneration}\label{area}

In this section, we are interested in the relationship between the Hilbert area and the degeneration of convex structres.  It is expected  if the convex structure diverges away from Teichm\"uller space, then the Hilbert area tends to infinity. Here are some supporting evidences.
\begin{theorem} Suppose the Hilbert area of a 
strictly convex real projective structure is  bounded. Then all the triangle invariants are  bounded in terms of the area.
\end{theorem}
\begin{proof}Let $\triangle$ be a lift of an ideal triangle in $\Omega$. Then the tangent lines at the vertices of
$\triangle$  to $\partial \Omega$ form a triangle $T$. Then
$$Area_{T}(\triangle)\leq Area_{\Omega}(\triangle)<R.$$
Up to affine transformation we can put $(T,\triangle)$ into the standard positions such that
the vertices of $T$ are $(0,1,1),(0,0,1),(1,0,1)$ and the vertices of $\triangle$ are
$(0,1/2,1),(1/2,1/2,1),(\alpha,0,1)$ where $0<\alpha\leq 1/2$. Then the triangle invariant $T(\cal E,\cal F,\cal G)$ of three flags
$$\cal E:\langle (0,1/2,1) \rangle \subset \langle (0,0,1), (0,1,1) \rangle $$
$$\cal F:\langle (1/2,1/2,1) \rangle \subset \langle (0,1,1), (1,0,1) \rangle $$
$$\cal G:\langle (\alpha,0,1) \rangle \subset \langle (0,0,1), (1,0,1) \rangle $$ is
$$\frac{1-\alpha}{\alpha}.$$
As $\alpha\ra 0$, $Area_{T}(\triangle(\alpha))\ra\infty$ (\cite{Col}), hence $\alpha\geq \epsilon= \epsilon(R)$.
This implies that
$$1\leq T(\cal E,\cal F,\cal G)\leq \frac{1-\epsilon}{\epsilon}.$$
\end{proof}
The above theorem also can be deduced from  \cite{AC}.

\begin{proposition}\label{bulging}
We fix a topological data of ideal triangulation of $\Sigma$ coming from pants decomposition to
define Bonahon-Dreyer coordinates. Suppose all triangle invariants are bounded and some shear invariant on an oriented infinite leaf $l$ satisfies that $|\sigma_1(l)-\sigma_2(l)|\ra\infty$. Then the Hilbert area of the projective surface $\Sigma$ defined by the Hilbert metric corresponding to Bonahon-Dreyer coordinates tends to infinity.
\end{proposition}
\begin{proof} First note that the difference of two shear parameters along a geodesic gives a bulging parameter. Hence if the difference of two shear parameters along $l$ tends to infinity, the part $D$ of the domain $\Omega$ delimited by $l$, and
the ideal triangle contained in $D$ whose one edege is $l$  converge to the triangle whose vertices are the end points $l(\pm\infty)$ of $l$ and the intersection $l^\perp$
of two tangent lines at $l(\pm\infty)$ if $v\ra\infty$. The area of such triangles tends to infinity. If $v\ra -\infty$, then $D$ degenerates to $l$, then the area of the  left ideal triangle tends to infinity. See \cite{Col} to see that the area is infinite for such a triangle. Hence the Hilbert area of one of two triangles sharing the infinite edge $l$ tends to infinity in the
limit domain. 
\end{proof}

On the other hand, even if $\sigma_1+\sigma_2 \ra\infty$ and the triangle invariants are bounded then
the area remains bounded since the sum represents the twisting parameter.

\section{Bonahon-Dreyer coordinates and Degeneration}\label{degeneration}
In this section, we give explicit Bonahon-Dreyer coordinates for the degenerated convex structures on a pair of pants $P$ and a punctured torus $T$.\\
{\bf Notations:} For a pair of pants $P$, three boundary components are $A_1,A_2,A_3$. Three positive eigenvalues of the monodromy of $A_i$ are $0<\lambda_i<\mu_i<\nu_i$. In Goldman parametrization, $\tau_i=\mu_i+\nu_i$. $P$ is decomposed into two triangles $T_\pm$ with edges $B_1,B_2,B_3$ which appear in the clock-wise order along $T_+$. See the Figure 1 in \cite{BK}. These edges are oriented counter-clockwise. The boundary component $A_i$ is opposite to $B_i$. The indices $i$ are modulo 3.
Goldman's internal parameter for a convex projective structure on $P$ are $s,t$ and an explicit formula can be found in \cite{BK}.
\subsection{Explicit formula for $P$}\label{formula}
Let $P$ be a pair of pants. The space of marked convex projective structures on $P$ with geodesic boundary has 8 dimensional parameter, $3\times 2$ boundary parameters, and 2 more internal parameters
$s$ and $t$ in Goldman parameters.
Now suppose we pinch all the boundary curves to cusps. Then all the boundary parameters vanish and
only 2 internal parameters survive in Goldman parameters. This can be justified by Figure 2 in \cite{BK} where verticies of $\Delta_+$ are parabolic fixed points.
This can be seen also by Proposition 4.1 of \cite{BK}. If all the boundary curves are parabolic,
then $\lambda_i=1, \tau_i=2$ where $\lambda_i$ is the smallest eigenvalue and $\tau_i$ is the sum of the rest eigenvalues of boundary curves. Then by the formula of Proposition 4.1 of \cite{BK},
$$\sigma_1(B_i)=\log (s\mu_{i-1}\sqrt{\frac{\lambda_{i-1}\lambda_{i+1}}{\lambda_i}})=\log s$$
$$ \sigma_2(B_i)=\log (\frac{\mu_{i+1}}{s}\sqrt{\frac{\lambda_{i-1}\lambda_{i+1}}{\lambda_i}})=-\log s $$
where $\mu_i=\frac{\tau_i-\sqrt{\tau_i^2-\frac{4}{\lambda_i}}}{2}.$
And $$\tau_{111}(T_+)=\log \frac{(e^{-\sigma_2(B_2)}+1)(e^{-\sigma_2(B_3)}+1)}{t(e^{\sigma_1(B_3)}+1)}=\log\frac{(s+1)}{t}.$$
\begin{proposition}\label{allparabolic}
Suppose $P$ is equipped with a convex projective structure with all three boundary components  being parabolic. Then two numbers $(\sigma_1(B_1), \tau_{111}(T_+))$ parametrize the space of such convex projective structures on $P$ with all three boundary components being parabolic.
\end{proposition}
  Similarly if one of the 3 boundary curves is pinched, then corresponding Bonahon-Dreyer
coordinates can be calculated as follows.
\begin{proposition}\label{oneparabolic}
Suppose $P$ is equipped with a convex projective structure with one boundary component $A_1$ being parabolic. Then Bonahon-Dreyer coordinates  $\sigma_1(B_1),\sigma_2(B_1), \sigma_1(B_2), \sigma_1(B_3), \tau_{111}(T^+),\tau_{111}(T^-)$ are  independent parameters to describe the space of
convex projective structures on $P$ with $A_1$ being parabolic.
\end{proposition}
\begin{proof}The holonomy of $A_1$ is parabolic, hence $\lambda_1=1,\tau_1=2$ and $\mu_1=1$.
Then 
$$ \sigma_1(B_2)=\log (s\sqrt{\frac{\lambda_{3}}{\lambda_2}})=  -\sigma_2(B_3)=-\log (\frac{1}{s}\sqrt{\frac{\lambda_{2}}{\lambda_3}}) .$$

It can be checked that
$$\sigma_1(B_1)-\sigma_2(B_1)=\log(s^2\frac{\mu_3}{\mu_2}),$$
$$\sigma_1(B_3)-\sigma_2(B_2)=\log(s^2\frac{\mu_2}{\mu_3}).$$
This gives
$$\sigma_1(B_1)-\sigma_2(B_1)+\sigma_1(B_3)-\sigma_2(B_2)=\log(s^4).$$
Hence there are 6 independent parameters.
\end{proof}

The codimension 1 boundary of $\cal P$ is the set of properly convex projective structures on $P$ with
one boundary being quasi-hyperbolic.
Recall that a quasi-hyperbolic isometry is a matrix in $SL(3,\R)$ which is conjugate to
$\begin{pmatrix}
 \mu & 1 & 0 \\
  0     & \mu & 0 \\
0 & 0 &  \nu\end{pmatrix}$ where $\mu>\nu>0, \mu^2\nu=1$.

Since quasi-hyperbolic elements are parametrized by one parameter only, if one of the boundary of $P$ is quasi-hyperbolic, such a projective structures is parametrized by 7 parameters in total.

\begin{proposition}\label{quasi}Suppose $P$ is equipped with a convex projective structure with one boundary component $A_1$ being quasi-hyperbolic. Then Bonahon-Dreyer coordinates   $$\sigma_1(B_1),\sigma_2(B_1),\sigma_1(B_2),\sigma_2(B_2),\sigma_1(B_3),\tau_{111}(T_\pm)$$ are  independent parameters to describe the space of
convex projective structures on $P$ with $A_1$ being quasi-hyperbolic.
\end{proposition}
\begin{proof} The holonomy of $A_1$ is quasi-hyperbolic, hence $\lambda_1=\nu, \tau_1=2\mu$
and $\mu_1=\mu$.   Then
$$\sigma_1(B_2)=\log(s\mu \sqrt{\frac{\nu\lambda_3}{\lambda_2}})$$
$$\sigma_2(B_3)=\log(\frac{\mu}{s}\sqrt{\frac{\nu\lambda_2}{\lambda_3}}).$$
Hence
$$\sigma_1(B_2)+\sigma_2(B_3)=2\log(\mu\sqrt{\nu})=0$$ since $\mu^2\nu=1$.
Then $\sigma_1(B_1),\sigma_2(B_1),\sigma_1(B_2),\sigma_2(B_2),\sigma_1(B_3),\tau_{111}(T_\pm)$ are  parameters.
\end{proof}

\subsection{Explicit formula for $T$}
In this section, we describe  explicit parameters to describe the convex real projective structures on a punctured torus $T$ whose boundary curve is $B$. Choose a meridian curve $C$ so that $T\setminus C$ is a pair of pants.  In this case, Goldman parameters are $\lambda(C),\tau(C),\lambda(B),\tau(B)$, two gluing parameters $(u,v)$ along $C$ together with two internal parameters $(s,t)$.
By Proposition 5.2 in \cite{BK}, Bonahon-Dreyer shear coordinates $\sigma_1(C)$ and $\sigma_2(C)$ are given by
$$ \sigma_1(C)=(u-u^0)-3 (v-v^0), \sigma_2(C)=(u-u^0)+ 3(v-v^0),$$ once a starting point $(u^0,v^0)$ is specified in Goldman's parameter. We normalize the coordinates that $(u^0,v^0)=(0,0)$.

Set $B=A_1, C=A_2=A_3$.
When $B$ becomes parabolic, $\lambda(B)=1,\tau(B)=2, \mu(B)=1$. Then by the proof of Theorem \ref{oneparabolic}, we get
$$\sigma_1(B_2)=-\sigma_2(B_3)$$ and
$$\sigma_1(B_1)-\sigma_2(B_1)=\log(s^2\frac{\mu_3}{\mu_2}),$$
$$\sigma_1(B_3)-\sigma_2(B_2)=\log(s^2\frac{\mu_2}{\mu_3}).$$ Furthermore
$$\sigma_1(B_1)-\sigma_2(B_1)+\sigma_1(B_3)-\sigma_2(B_2)=\log(s^4).$$
Hence 
 $\sigma_1(B_1),\sigma_2(B_1), \sigma_1(B_2), \sigma_1(B_3), \tau_{111}(T_+),\tau_{111}(T_-)$ are  independent parameters to describe  projective structures on a pair of pants when $B$ is parabolic.
But in our case, since $B=A_1, C=A_2=A_3$,  $\lambda_2=\lambda_3, \mu_2=\mu_3$.
Hence 
$$\sigma_1(B_2)=\log s=-\sigma_2(B_3), \sigma_1(B_1)-\sigma_2(B_1)=2\log s= \sigma_1(B_3)-\sigma_2(B_2).$$
Therefore if $\sigma_1(B_1), \sigma_1(B_2), \sigma_1(B_3)$ are given, then the others $$\sigma_2(B_3),\sigma_2(B_1)=\log(\frac{\mu_2\lambda_2}{s}), \sigma_2(B_2)=\log(\frac{\mu_2}{s})$$ are known.
Hence $s, \lambda_2=\lambda_3,\mu_2=\mu_3$ can be deduced from $$\sigma_1(B_1), \sigma_1(B_2), \sigma_1(B_3).$$

By the formula in Proposition 4.1 of \cite{BK},
$$ \tau_{111}(T_+)=\log \frac{(e^{-\sigma_2(B_2)}+1)(e^{-\sigma_2(B_3)}+1)}{t(e^{\sigma_1(B_3)}+1)}  $$
$$ \tau_{111}(T_-)=\log \frac{t\mu_1\mu_2\mu_3(e^{\sigma_1(B_3)}+1)}{(e^{-\sigma_2(B_2)}+1)(e^{-\sigma_2(B_3)}+1)}     .$$
Hence if $\sigma_1(B_1), \sigma_1(B_2), \sigma_1(B_3), \tau_{111}(T_+)$ are given, then $\tau_{111}(T_-)$ is known.
 Therefore,  $\sigma_1(B_1), \sigma_1(B_2), \sigma_1(B_3), \tau_{111}(T_+)$ parametrize such a structure on a pair of pants. When the pair of pants is glued along $C$, then two more shear parameters
$\sigma_1(C)=u-3v, \sigma_2(C)=u+3v$ are involved.
\begin{proposition}\label{parameterT}
The convex real projective structures on $T$ with the boundary curve being parabolic are parametrized by
$$\sigma_1(B_1), \sigma_1(B_2), \sigma_1(B_3), \tau_{111}(T_+)$$ and $\sigma_1(C),\sigma_2(C)$ where $C$ is a meridian curve so that $T\setminus C$ is a pair of pants.
\end{proposition}

When the boundary curve $B$ becomes quasi-hyperbolic, it determines a codimension 1 boundary of $\cal T$.
\begin{proposition}The space of convex real projective structures on $T$ with the boundary curve being quasi-hyperbolic, is parametrized by $\sigma_1(B_i),i=1,2,3, \sigma_1(C),\sigma_2(C), \tau_{111}(T_\pm).$
\end{proposition}
\begin{proof}From the proof of Theorem \ref{quasi}, if we set $B=A_1$, then $$\sigma_1(B_2)=-\sigma_2(B_3)$$ and $\sigma_1(B_2)=\log s$ using $\lambda_2=\lambda_3,\mu_2=\mu_3$.  Also we get
$$\sigma_1(B_1)=2\log s + \sigma_2(B_1), \sigma_1(B_3)=2\log s + \sigma_2(B_2).$$
Hence $\sigma_1(B_i),i=1,2,3, \tau_{111}(T_\pm), \sigma_1(C),\sigma_2(C)$ parametrize the space.
\end{proof}
For some related calculations and materials, see \cite{till}.

%\section{Fock-Goncharov coordiates for degeneration}

\vskip .1 in

\noindent     Inkang Kim\\
     School of Mathematics\\
     KIAS, Heogiro 85, Dongdaemen-gu\\
     Seoul, 02455, Korea\\
     \texttt{inkang\char`\@ kias.re.kr}
\end{document}